\documentclass[matann]{svjour}
\usepackage{amsmath}
\usepackage{amsxtra}
\usepackage{amscd}
\usepackage{amsthm}
\usepackage{amsfonts}
\usepackage{amssymb}
\usepackage{eucal}
\usepackage[matrix,arrow,curve]{xy}
\usepackage{scalerel}
\theoremstyle{plain}
\newtheorem{Thm}[subsection]{Theorem}
\newtheorem{Cor}[subsection]{Corollary}
\newtheorem{Prop}[subsection]{Proposition}
\newtheorem{Conj}[subsection]{Conjecture}

\theoremstyle{definition}
\newtheorem{Def}[subsection]{Definition}

\theoremstyle{remark}

\newtheorem{conj}{Conjecture}

\numberwithin{equation}{section}
\newif\ifShowLabels
\ShowLabelstrue
\newdimen\theight
\def\TeXref#1{%
    \leavevmode\vadjust{\setbox0=\hbox{{\tt
        \quad\quad  {\small \normalshape #1}}}%
    \theight=\ht0
    \advance\theight by \lineskip
    \kern -\theight \vbox to
    \theight{\rightline{\rlap{\box0}}%
    \vss}%
    }}%

\ShowLabelsfalse% comment this out if labels should be printed

\renewcommand{\sec}[2]{\section{#2}\label{S:#1}%
    \ifShowLabels \TeXref{{S:#1}} \fi}
\newcommand{\ssec}[2]{\subsection{#2}\label{SS:#1}%
    \ifShowLabels \TeXref{{SS:#1}} \fi}

\newcommand{\refs}[1]{Section ~\ref{S:#1}}
\newcommand{\refss}[1]{Section ~\ref{SS:#1}}

\newcommand{\reft}[1]{Theorem ~\ref{T:#1}}
\newcommand{\refp}[1]{Proposition ~\ref{P:#1}}
\newcommand{\refc}[1]{Corollary ~\ref{C:#1}}

\newcommand{\refe}[1]{\eqref{E:#1}}
\newcommand{\refco}[1]{Conjecture ~\ref{Co:#1}}

\newenvironment{thm}[1]%
    { \begin{Thm} \label{T:#1}  \ifShowLabels \TeXref{T:#1} \fi }%
    { \end{Thm} }

\renewcommand{\th}[1]{\begin{thm}{#1} \sl }
\renewcommand{\eth}{\end{thm} }

\newenvironment{propos}[1]%
    { \begin{Prop} \label{P:#1}  \ifShowLabels \TeXref{P:#1} \fi }%
    { \end{Prop} }
\newcommand{\prop}[1]{\begin{propos}{#1}\sl }
\newcommand{\eprop}{\end{propos}}

\newenvironment{corol}[1]%
    { \begin{Cor} \label{C:#1}  \ifShowLabels \TeXref{C:#1} \fi }%
    { \end{Cor} }
\newcommand{\cor}[1]{\begin{corol}{#1} \sl }
\newcommand{\ecor}{\end{corol}}

\newenvironment{defeni}[1]%
    { \begin{Def} \label{D:#1}  \ifShowLabels \TeXref{D:#1} \fi }%
    { \end{Def} }
\newcommand{\defe}[1]{\begin{defeni}{#1} \sl }
\newcommand{\edefe}{\end{defeni}}

\newenvironment{conjec}[1]%
    { \begin{Conj} \label{Co:#1}  \ifShowLabels \TeXref{Co:#1} \fi }%
    { \end{Conj} }
\renewcommand{\conj}[1]{\begin{conjec}{#1} \sl }
\newcommand{\econj}{\end{conjec}}

\newcommand{\eq}[1]%
    { \ifShowLabels \TeXref{E:#1} \fi
       \begin{equation} \label{E:#1} }
\newcommand{\eeq}{ \end{equation} }

\newcommand\gam{\gamma}     \newcommand\Gam{\Gamma}

\newcommand\lam{\lambda}        \newcommand\Lam{\Lambda}

\newcommand\calM{{\mathcal{M}}}

\newcommand\calO{{\mathcal{O}}}

\newcommand\calW{{\mathcal{W}}}

        \newcommand\bfT{{\mathbf T}}

\newcommand\PP{\mathbb{P}}
\renewcommand\AA{\mathbb{A}}

\newcommand\ZZ{\mathbb{Z}}

\newcommand\CC{\mathbb{C}}

 \newcommand\grg{{\mathfrak{g}}}

 \newcommand\grt{{\mathfrak{t}}}

\newcommand\sdp{\times \hskip -0.3em {\raise 0.3ex
\hbox{$\scriptscriptstyle |$}}} % semidirect product

\newcommand\End{\operatorname{End\,}}

\newcommand\Gr{\operatorname{Gr}}

\newcommand\Int{\operatorname{Int}}

\newcommand\SL{{\normalshape SL}}

\newcommand\Sym{\operatorname{Sym}}

\newcommand\hatPsi{{\widehat{\Psi}}}

\newcommand\tilR{{\widetilde{R}}}

\newcommand\x{\times}

\newcommand{\la}{\langle}

\newcommand\nc{\newcommand}

\nc\aff{\operatorname{aff}}
\nc\oGr{\overline{\Gr}}
\nc\Bun{\operatorname{Bun}}
\nc\hgrg{\widehat{\grg}}
\renewcommand\Int{\operatorname{Int}}
\nc\bInt{\overline{\Int}}
\nc\hatLam{\widehat{\Lam}}
\nc\bmu{\overline{\mu}}
\nc\bnu{\overline{\nu}}
\nc\blambda{\overline{\lam}}
\renewcommand\SL{\operatorname{SL}}
\nc\ocalW{\overline{\calW}}
\nc\pos{\operatorname{pos}}
\nc\IH{\operatorname{IH}}
\nc\Rep{\operatorname{Rep}}
\nc\Gal{\operatorname{Gal}}
\nc{\tilGr}{\widetilde{\Gr}}

\nc\Pic{\operatorname{Pic}}
\nc{\HC}{{\mathcal{HC}}}
\nc{\on}{\operatorname}
\nc{\BA}{{\mathbb{A}}}
\nc{\BC}{{\mathbb{C}}}
\nc{\BG}{{\mathbb{G}}}
\nc{\BM}{{\mathbb{M}}}
\nc{\BN}{{\mathbb{N}}}
\nc{\BQ}{{\mathbb{Q}}}
\nc{\BP}{{\mathbb{P}}}
\nc{\BR}{{\mathbb{R}}}
\nc{\BZ}{{\mathbb{Z}}}
\nc{\BS}{{\mathbb{S}}}

\nc{\CA}{{\mathcal{A}}}
\nc{\CB}{{\mathcal{B}}}
\nc{\CalC}{{\mathcal C}}
\nc{\CalD}{{\mathcal D}}
\nc{\CE}{{\mathcal{E}}}
\nc{\CF}{{\mathcal{F}}}
\nc{\CG}{{\mathcal{G}}}
\nc{\CH}{{\mathcal{H}}}
\nc{\CK}{{\mathcal{K}}}
\nc{\CL}{{\mathcal{L}}}
\nc{\CM}{{\mathcal{M}}}
\nc{\CMM}{{\mathcal{M}^{\operatorname{gen}}_\hbar(-\rho)}}
\nc{\CN}{{\mathcal{N}}}
\nc{\CO}{{\mathcal{O}}}
\nc{\CP}{{\mathcal{P}}}
\nc{\CQ}{{\mathcal{Q}}}
\nc{\CR}{{\mathcal{R}}}
\nc{\CS}{{\mathcal{S}}}
\nc{\CT}{{\mathcal{T}}}
\nc{\CU}{{\mathcal{U}}}
\nc{\CV}{{\mathcal{V}}}
\nc{\CW}{{\mathcal{W}}}
\nc{\CX}{{\mathcal{X}}}
\nc{\CY}{{\mathcal{Y}}}
\nc{\CZ}{{\mathcal{Z}}}

\nc{\gen}{{\operatorname{gen}}}
\nc{\cM}{{\check{\mathcal M}}{}}
\nc{\csM}{{\check{\mathcal A}}{}}
\nc{\obM}{{\overset{\circ}{\mathbf M}}{}}
\nc{\oCA}{{\overset{\circ}{\mathcal A}}{}}
\nc{\obA}{{\overset{\circ}{\mathbf A}}{}}
\nc{\ooM}{{\overset{\circ}{M}}{}}
\nc{\osM}{{\overset{\circ}{\mathsf M}}{}}
\nc{\vM}{{\overset{\bullet}{\mathcal M}}{}}
\nc{\nM}{{\underset{\bullet}{\mathcal M}}{}}
\nc{\obD}{{\overset{\circ}{\mathbf D}}{}}
\nc{\cp}{{\overset{\circ}{\mathbf p}}{}}
\nc{\ofZ}{{\overset{\circ}{\mathfrak Z}}{}}

\nc{\fa}{{\mathfrak{a}}}
\nc{\fb}{{\mathfrak{b}}}
\nc{\fg}{{\mathfrak{g}}}
\nc{\fgl}{{\mathfrak{gl}}}
\nc{\fh}{{\mathfrak{h}}}
\nc{\fj}{{\mathfrak{j}}}
\nc{\fm}{{\mathfrak{m}}}
\nc{\fn}{{\mathfrak{n}}}
\nc{\fu}{{\mathfrak{u}}}
\nc{\fp}{{\mathfrak{p}}}
\nc{\frr}{{\mathfrak{r}}}
\nc{\fs}{{\mathfrak{s}}}
\nc{\ft}{{\mathfrak{t}}}
\nc{\fT}{{\mathfrak{T}}}
\nc{\ofT}{{\overline{\mathfrak T}}}
\nc{\ofS}{{\overline{\mathfrak S}}}
\nc{\fsl}{{\mathfrak{sl}}}
\nc{\hsl}{{\widehat{\mathfrak{sl}}}}
\nc{\hgl}{{\widehat{\mathfrak{gl}}}}
\nc{\hg}{{\widehat{\mathfrak{g}}}}
\nc{\chg}{{\widehat{\mathfrak{g}}}{}^\vee}
\nc{\hn}{{\widehat{\mathfrak{n}}}}
\nc{\chn}{{\widehat{\mathfrak{n}}}{}^\vee}

\nc{\fA}{{\mathfrak{A}}}
\nc{\fB}{{\mathfrak{B}}}
\nc{\fD}{{\mathfrak{D}}}
\nc{\fE}{{\mathfrak{E}}}
\nc{\fF}{{\mathfrak{F}}}
\nc{\fG}{{\mathfrak{G}}}
\nc{\fI}{{\mathfrak{I}}}
\nc{\fJ}{{\mathfrak{J}}}
\nc{\fK}{{\mathfrak{K}}}
\nc{\fL}{{\mathfrak{L}}}
\nc{\fM}{{\mathfrak{M}}}
\nc{\fN}{{\mathfrak{N}}}
\nc{\frP}{{\mathfrak{P}}}
\nc{\fQ}{{\mathfrak Q}}
\nc{\fS}{{\mathfrak S}}
\nc{\fU}{{\mathfrak{U}}}
\nc{\fZ}{{\mathfrak{Z}}}

\nc{\bb}{{\mathbf{b}}}
\nc{\bc}{{\mathbf{c}}}
\nc{\be}{{\mathbf{e}}}
\nc{\bj}{{\mathbf{j}}}
\nc{\bn}{{\mathbf{n}}}
\nc{\bp}{{\mathbf{p}}}
\nc{\bq}{{\mathbf{q}}}
\nc{\bv}{{\mathbf{v}}}
\nc{\bx}{{\mathbf{x}}}
\nc{\by}{{\mathbf{y}}}
\nc{\bw}{{\mathbf{w}}}
\nc{\bA}{{\mathbf{A}}}
\nc{\bB}{{\mathbf{B}}}
\nc{\bC}{{\mathbf{C}}}
\nc{\bK}{{\mathbf{K}}}
\nc{\bD}{{\mathbf{D}}}
\nc{\bH}{{\mathbf{H}}}
\nc{\bM}{{\mathbf{M}}}
\nc{\bN}{{\mathbf{N}}}
\nc{\bO}{{\mathbf{O}}}
\nc{\bQ}{{\mathbf Q}}
\nc{\bS}{{\mathbf{S}}}
\nc{\bT}{{\mathbf{T}}}
\nc{\bV}{{\mathbf{V}}}
\nc{\bW}{{\mathbf{W}}}
\nc{\bX}{{\mathbf{X}}}
\nc{\bP}{{\mathbf{P}}}
\nc{\bZ}{{\mathbf{Z}}}

\nc{\sA}{{\mathsf{A}}}
\nc{\sB}{{\mathsf{B}}}
\nc{\sC}{{\mathsf{C}}}
\nc{\sD}{{\mathsf{D}}}
\nc{\sF}{{\mathsf{F}}}
\nc{\sK}{{\mathsf{K}}}
\nc{\sM}{{\mathsf{M}}}
\nc{\sO}{{\mathsf{O}}}
\nc{\sQ}{{\mathsf{Q}}}
\nc{\sP}{{\mathsf{P}}}
\nc{\sV}{{\mathsf{V}}}
\nc{\sW}{{\mathsf{W}}}
\nc{\sZ}{{\mathsf{Z}}}
\nc{\sfp}{{\mathsf{p}}}
\nc{\sr}{{\mathsf{r}}}
\nc{\st}{{\mathsf{t}}}
\nc{\sfb}{{\mathsf{b}}}
\nc{\sfc}{{\mathsf{c}}}
\nc{\sd}{{\mathsf{d}}}
\nc{\sg}{{\mathsf{g}}}
\nc{\sfl}{{\mathsf{l}}}

\nc{\BK}{{\bar{K}}}

\nc{\tA}{{\widetilde{\mathbf{A}}}}
\nc{\tB}{{\widetilde{\mathcal{B}}}}
\nc{\tg}{{\widetilde{\mathfrak{g}}}}
\nc{\tG}{{\widetilde{G}}}
\nc{\TM}{{\widetilde{\mathbb{M}}}{}}
\nc{\tO}{{\widetilde{\mathsf{O}}}{}}
\nc{\tU}{{\widetilde{\mathfrak{U}}}{}}
\nc{\TZ}{{\tilde{Z}}}
\nc{\tZ}{\widetilde{Z}{}}
\nc{\tx}{{\tilde{x}}}
\nc{\tbv}{{\tilde{\bv}}}
\nc{\tfP}{{\widetilde{\mathfrak{P}}}{}}
\nc{\tz}{{\tilde{\zeta}}}
\nc{\tmu}{{\tilde{\mu}}}

\nc{\td}{\ddot{\underline{d}}{}}
\nc{\tzeta}{\widetilde{\zeta}{}}
\nc{\hd}{{\widehat{\underline{d}}}}
\nc{\hG}{{\widehat{G}}}
\nc{\hBP}{\widehat{\mathbb P}{}}
\nc{\hQ}{{\widehat{Q}}}
\nc{\hsM}{\widehat{\mathsf M}{}}
\nc{\hfM}{\widehat{\mathfrak M}{}}
\nc{\hCP}{\widehat{\mathcal P}{}}
\nc{\hCR}{\widehat{\mathcal R}{}}
\nc{\hCS}{{\widehat{\mathcal S}}}
\nc{\hfZ}{\widehat{\mathfrak Z}{}}

\nc{\urho}{\underline{\rho}}
\nc{\uB}{\underline{B}}
\nc{\uC}{{\underline{\mathbb{C}}}}
\nc{\ui}{\underline{i}}
\nc{\ofP}{{\overline{\mathfrak{P}}}}
\nc{\hrho}{{\hat{\rho}}}

\nc{\unl}{\underline}
\nc{\ol}{\overline}
\nc{\one}{{\mathbf{1}}}
\nc{\two}{{\mathbf{t}}}

\nc{\Tot}{{\mathop{\operatorname{\normalshape Tot}}}}
\nc{\Hilb}{{\mathop{\operatorname{\normalshape Hilb}}}}
\nc{\CHom}{{\mathop{\operatorname{{\mathcal{H}}\it om}}}}
\nc{\defi}{{\mathop{\operatorname{\normalshape def}}}}
\nc{\length}{{\mathop{\operatorname{\normalshape length}}}}
\nc{\Cliff}{{\mathsf{Cliff}}}
\nc{\Fl}{{\mathsf{Fl}}}
\nc{\Fib}{{\mathsf{Fib}}}
\nc{\Coh}{{\mathsf{Coh}}}
\nc{\FCoh}{{\mathsf{FCoh}}}

\nc{\reg}{{\text{\normalshape reg}}}

\nc{\cplus}{{\mathbf{C}_+}}
\nc{\cminus}{{\mathbf{C}_-}}
\nc{\cthree}{{\mathbf{C}_*}}
\nc{\Qbar}{{\bar{Q}}}

\nc{\bh}{{\bar{h}}}
\nc{\bOmega}{{\overline{\Omega}}}
\nc\tGr{\widetilde{\Gr}}

\nc{\seq}[1]{\stackrel{#1}{\sim}}
\nc\ogu{\overline{G/U}}
\nc\chlam{\check{\lam}}
\nc\St{\operatorname{St}}

\nc\uS{\underline{S}}
\nc\QM{\mathcal{QM}}
\newcommand{\oQM}{\vphantom{j^{X^2}}\smash{\overset{\circ}{\vphantom{\vstretch{0.7}{A}}\smash{\QM}}}}
\nc{\chmu}{\check{\mu}}
\begin{document}
\title{Weyl modules and $q$-Whittaker functions}
\dedication{To the memory of Andrei Zelevinsky who taught us the beauty of
symmetric functions}
\author{Alexander Braverman\inst{1} \and Michael Finkelberg\inst{2}}
\institute{Department of Mathematics, Brown University,
151 Thayer St., Providence RI 02912, USA, \email{braval@math.brown.edu} \and
IMU, IITP and National Research University
Higher School of Economics
Department of Mathematics, 20 Myasnitskaya st, Moscow 101000 Russia,
\email{fnklberg@gmail.com}}
\date{Received: / Revised version: }
%\offprints{Michael Finkelberg}
%\mail{National Research University Higher School of Economics
%Department of Mathematics, 20 Myasnitskaya st, Moscow 101000 Russia}
\maketitle
\begin{abstract}
Let $G$ be a semi-simple simply connected group over $\CC$. Following \cite{GLOI} we use the $q$-Toda integrable system obtained by quantum group version
of the
Kostant-Whittaker reduction (cf. \cite{Et} and \cite{Sev}) to define the notion of $q$-Whittaker functions $\Psi_{\chlam}(q,z)$.
This is a family of invariant polynomials on the maximal torus $T\subset G$ (here $z\in T$) depending
on a dominant weight $\check\lam$ of $G$ whose coefficients are rational functions in a variable $q\in \CC^*$.
For a conjecturally the same (but a priori different) definition of the $q$-Toda system these functions were studied
by B.~Ion in \cite{Ion} and by I.~Cherednik in \cite{C} (we shall denote the $q$-Whittaker functions from \cite{C}
by $\Psi'_{\chlam}(q,z)$). For $G=SL(N)$ these functions were extensively studied in \cite{GLOI}-\cite{GLOIII}.

We show that when $G$ is simply laced, the function $\hatPsi_{\check\lam}(q,z)=\Psi_{\chlam}(q,z)\cdot{\prod\limits_{i\in I}\prod\limits_{r=1}^{\langle\alpha_i,\check\lambda\rangle}(1-q^r)}$
(here $I$ denotes the set of vertices of the Dynkin diagram
of $G$) is equal to the character of a certain finite-dimensional $G[[\st]]\rtimes\CC^*$-module $D(\check\lam)$
(the Demazure module). When $G$ is not simply laced a twisted version of the above statement holds. This result is known for $\Psi_{\chlam}$ replaced by $\Psi'_{\chlam}$
(cf. \cite{San} and \cite{Ion}); however our proofs are algebro-geometric (and rely on our previous work \cite{BF11}) and thus they are completely
different from \cite{San} and
\cite{Ion} (in particular, we give an apparently new algebro-geometric interpretation of the modules $D(\chlam)$).
\end{abstract}
%\maketitle

\sec{int}{Introduction}
%----------------------------------------------------------------------------------------------------------
\ssec{}{The $q$-Whittaker functions}Let $G$ be a semi-simple, simply connected group over $\CC$ with Lie algebra
$\grg$; we choose a pair of opposite Borel subgroups $B, B_-$ of $G$ with unipotent radicals
$U, U_-$; the intersection $B\cap B_-$ is a maximal torus $T$ of $G$. It will be convenient for us to denote
the weight lattice of $T$ by $\check \Lam$ and the coweight lattice by $\Lam$. In this paper
we study certain invariant polynomials $\Psi_{\chlam}(q,z)$ on $T$ (the invariance is with respect
to the Weyl group $W$ of $G$). Here $z\in T$, $q\in \CC^*$ and $\check\lam:T\to \CC^*$ is a dominant weight of $G$. The function
 $\Psi_{\chlam}(q,z)$ is a polynomial function of $z$ with coefficients which are rational functions of $q$ (in fact, later
 were are going to work with a certain modification $\hatPsi_{\chlam}(q,z)$ of $\Psi_{\chlam}(q,z)$ which will be polynomial in
$q$).

The definition of $\Psi_{\chlam}(q,z)$ is as follows. Let $\check{G}$ denote the Langlands dual
group of $G$ with its maximal torus $\check{T}$. In \cite{Et} and \cite{Sev} the authors define (by adapting the so called
Kostant-Whittaker reduction to the case of quantum groups) a homomorphism
$\calM:\CC[T]^W\to \End_{\CC(q)}\CC(q)[\check T]$ called the quantum difference Toda integrable system associated with
$\check G$.
For each $f\in \CC[T]^W$ the operator $\calM_f:=\calM(f)$ is indeed a difference operator: it is a $\CC(q)$-linear combination of
shift operators $\bfT_{\check\beta}$ where $\check\beta\in\check\Lam$ and
$$
\bfT_{\check\beta}(F(x))=F(q^{\check\beta}x).
$$

\noindent
{\bf Remark.} In principle the constructions of \cite{Et} and \cite{Sev} depend on a choice of orientation of the Dynkin diagram of
$\check{G}$; however one can deduce from the main result of \cite{ferm} that the resulting homomorphism is independent of this choice.

\medskip
\noindent
In particular, the above operators can be restricted to operators acting in the space of functions on the lattice
$\check\Lam$ by means
of the embedding $\check\Lam\hookrightarrow \check T$ sending every $\chlam$ to $q^{\chlam}$. For any $f\in \CC[T]^W$
we shall
denote the corresponding operator by $\calM_f^{\on{lat}}$. The following conjecture should probably be not very difficult; however, at the moment we don't know how to prove it:
\conj{whit-exist-co}
\begin{enumerate}
\item
There exists a unique collection of $\CC(q)$-valued polynomials $\Psi_{\chlam}(q,z)$ on $T$ satisfying the following properties:

a) $\Psi_{\chlam}(q,z)=0$ if $\chlam$ is not dominant.

b) $\Psi_{0}(q,z)=1$.

c) Let us consider all the functions $\Psi_{\chlam}(q,z)$ as one function $\Psi(q,z):\check \Lam\to \CC(q)$ depending
on $z\in  T$. Then for every $f\in \CC[T]^W$ we have
$$
\calM^{\on{lat}}_f(\Psi(q,z))=f(z)\Psi(q,z).
$$
\item The polynomials $\Psi_{\chlam}(q,z)$ are $W$-invariant.
\end{enumerate}
\econj
Of course, the second statement follows from the ``uniqueness" part of the first.

Some remarks about the literature are necessary here. First of all, \refco{whit-exist-co} is easy for $G=SL(N)$. In this case, the functions
$\Psi_{\chlam}(q,z)$ are extensively studied in \cite{GLOI}-\cite{GLOIII}. Second, for general $G$ there exists another definition of the
$q$-Toda system using double affine Hecke algebras, studied for example in \cite{C}. Since it is not clear to us
how to prove that the definition of $q$-Toda from \cite{C} and the definition of \cite{Et} and \cite{Sev} are the same,
we shall denote the operators from \cite{C} by $\calM_f'$. It is easy to see
that $\calM_f=\calM_f'$ for $G=SL(N)$.\footnote{In fact, as we are going to explain later, the results of this
paper together with the results of \cite{Ion} imply that $\calM_f=\calM_f'$
for any $G$, but we would like to have a more direct proof of this fact.}
Similarly we shall denote by $(\calM^{\on{lat}}_f)'$ their ``lattice'' version. Then it is shown in \cite{C} that the existence part of
\refco{whit-exist-co} holds for any $G$ if the operators $\calM^{\on{lat}}_f$ are replaced by $(\calM^{\on{lat}}_f)'$.
We shall denote the corresponding polynomials by $\Psi_{\chlam}'(q,z)$.

The main result of this paper  will imply the following:
\th{whit-exist}
\begin{enumerate}
\item
There exists a collection of $W$-invariant polynomials $\Psi_{\chlam}(q,z)$ on $T$ with coefficients in $\CC(q)$ satisfying a), b) and c) above.
\item
Let $\hatPsi_{\check\lam}(q,z)=\Psi_{\chlam}(q,z)\cdot{\prod\limits_{i\in I}\prod\limits_{r=1}^{\langle\alpha_i,\check\lambda\rangle}(1-q^r)}$. Then $\hatPsi_{\chlam}(q,z)$ is a polynomial
    function on $\AA^1\x  T$.
\end{enumerate}
\eth
We are going to construct the above polynomials explicitly by algebro-geometric
means. Thus we prove the existence part of~\refco{whit-exist-co}.

We shall usually refer to the polynomials $\Psi_{\chlam}$ and $\hatPsi_{\chlam}$ as $q$-Whittaker functions (following
\cite{GLOI}-\cite{GLOIII}).
It is not difficult to see that
$$
\lim\limits_{q\to 0}\Psi_{\chlam}=\lim\limits_{q\to 0}\hatPsi_{\chlam}=\chi(L(\chlam))
$$
where $\chi(L(\lam))$ stands for the character of the irreducible representation $L(\chlam)$ of $G$ with highest
weight $\chlam$.

The main purpose of this paper is to give several (algebro-geometric and representation-theoretic) interpretations of
the functions $\Psi_{\chlam}$ and $\hatPsi_{\chlam}$; as a byproduct we shall show that $\hatPsi_{\chlam}(q,z)$ is {\em positive},
i.e. it is a linear combination of the functions $\chi(L(\chmu))$ with coefficients in $\ZZ_{\geq 0}[q]$ (this also implies that
$\Psi_{\chlam}$ is a linear combination of the $\chi(L(\chmu))$'s with coefficients in $\ZZ_{\geq 0}[[q]]$). All of our results are known
for the polynomials $\Psi_{\chlam}'$  (and thus, in particular, we can show that $\Psi_{\chlam}=\Psi_{\chlam}'$) due to
\cite{C} and~\cite{San},~\cite{Ion} but our proofs are totally different from {\em loc. cit.}
%-------------------------------------------------------------------------------------------------------------------------------
\ssec{}{Weyl modules}
Recall the notion of Weyl $\fg[\st]$-module $\CW(\check\lambda)$
for dominant $\check\lambda\in\Lambda^\vee_+$, see e.g.~\cite{CFK}.
It is the maximal $G$-integrable $\fg[\st]$-quotient module of
$\on{Ind}_{\fu[\st]\oplus\ft}^{\fg[\st]}\BC_{\check\lambda}$
where $\fu\subset \fg$ is the nilpotent radical of a Borel subalgebra, containing $\grt$. There is also a natural
notion of {\em dual Weyl module} $\calW(\chlam)^{\vee}$ (one has to replace the induction by coinduction and ``quotient module"
by ``submodule"). Both $\calW(\chlam)$ and $\calW(\chlam)^{\vee}$ are endowed with a natural action of
$\CC^*$ by ``loop rotation". When restricted to $G\x \CC^*$ the module
$\calW(\chlam)$ becomes a direct sum of finite-dimensional representations and the character $\chi(\calW(\chlam))$ makes sense;
moreover it is a linear combination of $\chi(L(\chmu))$'s with coefficients in $\ZZ_{\geq 0}[[q]]$. Also we have
$\chi(\calW(\chlam))=\chi(\calW(\chlam)^{\vee})$.

Let $\AA^{\chlam}$ denote the space of all formal linear combinations $\sum\gam_i x_i$ where
$x_i\in \AA^1$ and $\gam_i$ are dominant weights of $G$ such that $\sum \gam_i=\chlam$.
The character of $\CC[\AA^{\chlam}]$ with respect to the natural action of
$\CC^*$ is equal to ${\prod\limits_{i\in I}\prod\limits_{r=1}^{\langle\alpha_i,\check\lambda\rangle}(1-q^r)}$. According to~\cite{CFK} there exists
an action of $\CC[\AA^{\chlam}]$ on $\calW(\chlam)$ such that

1) This action commutes with $G[\st]\rtimes \CC^*$;

2) $\calW(\chlam)$ is finitely generated and free
over $\CC[\AA^{\chlam}]$.

Let $D(\chlam)$ be the fiber of $\calW(\chlam)$ at $\chlam\cdot 0\in \AA^{\chlam}$. This module is called
a Demazure module (for reasons explained in~\cite{CL} and~\cite{FL}).
This is a finite-dimensional $G[\st]\rtimes \CC^*$-module
(in fact, it is easy to see that the action of $G[\st]$ on $D(\chlam)$ extends to an action of $G[[\st]]$).
We are going to prove the following

\th{demazure}Assume that $G$ is simply laced. Then
\begin{enumerate}
\item
\eq{dem-1}
\chi(\calW(\chlam))=\Psi_{\chlam}(q,z)
\end{equation}
\item
\eq{dem-2}
\chi(D(\chlam))=\hatPsi_{\chlam}(q,z).
\end{equation}
\end{enumerate}
In particular, $\hatPsi_{\chlam}(q,z)$ is positive in the sense discussed above.
\eth
When $G$ is not simply laced, the above result is still true, if one replaces $G[[\st]]$ by some twisted (in the sense of Kac-Moody groups) version
of it; we shall not give the details here (cf. \refss{non-simp} for a discussion of the non-simply laced case).

\reft{demazure}(2) is proved in \cite{Ion} for $\hatPsi_{\chlam}'$ instead of $\hatPsi_{\chlam}$.
\footnote{It is important to emphasize that the definition of Demazure modules used in this paper (as fibers of Weyl modules)
is not obviously equivalent to the standard definition used in \cite{Ion}; however, the equivalence of the two definitions
is proved in~\cite{CL} in type $A$, and in~\cite{FL} in general.}
Thus \reft{demazure} together with \cite{Ion} imply the following
\cor{twotoda}
Assume that $G$ is simply laced. Then we have $\hatPsi_{\chlam}'=\hatPsi_{\chlam}$. Hence
for any $f\in \CC[T]^W$ we have $\calM_f=\calM'_f$.
\ecor
As was mentioned earlier we would like to have a more direct proof of this result
(independent of the results of \cite{Ion} and this paper).
We would also like to emphasize that our proof of \reft{demazure} is geometric (in fact it follows easily from the main
result of \cite{BF11}) and thus it is quite different from the proof in \cite{Ion}. Also, \refc{twotoda} is
wrong if $G$ is not simply laced, cf.~\refss{non-simp}.
%----------------------------------------------------------------------------------------------------------------------------------
\ssec{}{Geometric interpretation and spaces of (quasi-)maps}
To prove \reft{demazure} it is clearly enough to prove \refe{dem-1}.
This will be done by interpreting both the LHS and the RHS in terms of algebraic geometry.

Let us first do it for the LHS. The quotient $G[[\st]]/T\cdot U_-[[\st]]$
can naturally be regarded as a scheme
over $\CC$. Any weight $\chlam$ defines a $G[[\st]]\rtimes \CC^*$-equivariant line bundle on this scheme in the standard way.
We shall prove
\th{weyl}
There is a natural isomorphism $\Gam(G[[\st]]/T\cdot U_-[[\st]],\calO(\chlam))\simeq \calW(\chlam)^{\vee}$. Similarly,
$\Gam(G[[\st]]/B_-[[\st]],\calO(\chlam))\simeq D(\chlam)^{\vee}$.
\eth
{\bf Remark.} \reft{weyl} is not difficult; it can be thought of as an analog of Borel-Weil-Bott theorem for $G[[\st]]$. Let us
also stress, that while the dual Weyl module $\calW(\chlam)^{\vee}$ has a natural action of $G[[\st]]$, the Weyl module
$\calW(\chlam)$ itself only has an action of $G[\st]$.

\medskip
\noindent
On the other hand, there is a well known connection between the quotient $G[[\st]]/T\cdot U_-[[\st]]$ and the space of based maps
$\PP^1\to G/B$. Moreover, in \cite{BF11} we have given a construction of the universal eigen-function of the operators
$\calM_f$ via the geometry of the above spaces of maps. Using this construction, we can obtain \refe{dem-1} from \reft{weyl}
by a (simple) sequence of formal manipulations. Technically, in order to perform this we shall need to consider a
compactification of the space of maps by the corresponding space of quasi-maps.
\ssec{non-simp}{The case of non-simply laced $G$}
Formally, the above results do not hold when $G$ is not simply laced. However, it is
easy to adjust all the results to the non-simply laced case following~Section~7 of~\cite{BF11}; in particular, in the non-simply laced case the
functions $\Psi_{\chlam}$ and $\hatPsi_{\chlam}$ should be interpreted as
the characters of global (resp. local) Weyl modules for the distinguished
maximal parahoric subalgebra in a certain twisted affine algebra corresponding
to $\fg$ (cf.~Section~7 of~\cite{BF11} for more detail). The relevant theory
of Weyl modules and their relation to Demazure
modules in the twisted case is developed in \cite{FoKu}.
On the other hand, the character of {\em nontwisted} local Weyl modules are
identified with $\hatPsi'_{\chlam}$ in~\cite{Len}.
%-------------------------------------------------------------------------------------------------------------------
\ssec{}{Plan of the paper} This paper is organized as follows.
 In \refs{2} we discuss certain line
bundles on the space of (quasi-)maps and relate those to sections of a line bundle on $G[[\st]]/T\cdot U_-[[\st]]$.
\refs{3} is devoted to the proof of certain cohomology vanishing on the space of quasi-maps.
In \refs{4} we give an interpretation of $\Psi_{\chlam}$ via quasi-maps. Finally in \refs{5} we give a proof of
\reft{demazure}.
\sec{2}{Quasimaps' scheme}

We follow the notations of~\cite{BF11}, unless specified otherwise.

\ssec{ind}{Ind-scheme $\fQ$}
Given $\beta\geq\alpha\in\Lambda_+$ (the cone of positive integral combinations
of the simple coroots) we consider the closed embedding
$\varphi_{\alpha,\beta}:\ \QM_\fg^\alpha\hookrightarrow\QM_\fg^\beta$ adding the
defect $(\beta-\alpha)\cdot0$ at the point $0\in\bC$. We denote by
$\fQ$ the direct limit of this system.

Recall that $V_{\check\omega_i},\ i\in I$, are the fundamental $\fg$-modules,
and $\QM^\alpha_\fg$ is equipped with a closed embedding
$\psi_\alpha:\ \QM^\alpha_\fg\hookrightarrow
\prod_{i\in I}\BP\Gamma(\bC,V_{\check\omega_i}\otimes\CO(\langle\alpha,
\check\omega_i\rangle))$.
Given a $\fg$-weight $\check\lambda=
\sum_{i\in I}d_i\check\omega_i\in\Lambda^\vee$ we define a line bundle
$\CO(\check\lambda)^\alpha$ on $\QM^\alpha_\fg$ as
$\psi_\alpha^*\bigotimes_{i\in I}\CO(d_i)$.
Note that if $\check\lambda$ is dominant, i.e. $d_i\geq0\ \forall i$,
then $\CO(\check\lambda)^\alpha$ is the inverse image of $\CO(1)$ on
$\BP\Gamma(\bC,V_{\check\lambda}\otimes\CO(\langle\alpha,
\check\lambda\rangle))$ under the natural morphism
$\QM^\alpha_\fg\to\BP\Gamma(\bC,V_{\check\lambda}\otimes\CO(\langle\alpha,
\check\lambda\rangle))$.
Clearly,
$\varphi_{\alpha,\beta}^*\CO(\check\lambda)^\beta\simeq\CO(\check\lambda)^\alpha$.
The resulting line bundle on the ind-scheme $\fQ$ is denoted
$\CO(\check\lambda)$.

\ssec{pro}{Infinite type scheme $\bQ$}
We denote $\BC[[t^{-1}]]$ by $R$, and $\BC((t^{-1}))$ by $F$. Recall that
$R_n=R/(t^{-n})$. We denote the projection $R\twoheadrightarrow R_n$ by $p_n$.
For a $\BC$-algebra $A$, the $A$-points of the infinite type scheme 
$\overline{G/U_-}(R)$ are the collections of
vectors $v_{\check\lambda}\in V_{\check\lambda}\otimes A[[t^{-1}]],\ \check\lambda\in
\Lambda^\vee_+$ (dominant $\fg$-weights), satisfying the Pl\"ucker equations.
We denote by $\underline{\widehat\bQ}\subset\overline{G/U_-}(R)$ the open subscheme
formed by all the maps $\on{Spec}R\to\overline{G/U_-}$ whose restriction
to the generic point of $\on{Spec}R$ lands into
$G/U_-\subset\overline{G/U_-}(R)$. According to~\cite{fema},
$\underline{\widehat\bQ}$ is {\em not reduced} in general. We define 
$\widehat\bQ:=(\underline{\widehat\bQ})_{\operatorname{red}}$.
%For $n\geq1$ we consider the open subscheme $\overline{G/U_-}(R)_n$ formed
%by all the collections $(v_{\check\lambda}\in V_{\check\lambda}\otimes R)$ such
%that $p_n(v_{\check\omega_i})\ne0$ for all $i\in I$. Clearly,
%$\overline{G/U_-}(R)_1=G/U_-(R)$. We denote the union of all
%$\overline{G/U_-}(R)_n$ by $\widehat\bQ$.
It is equipped with a free action
of the Cartan torus $T:\ h(v_{\check\lambda})=\check\lambda(h)v_{\check\lambda}$.
The quotient scheme $\underline\bQ=\underline{\widehat\bQ}/T$ is a closed 
subscheme in
$\prod_{i\in I}\BP(V_{\check\omega_i}\otimes R)$. We define 
$\bQ:=(\underline\bQ)_{\operatorname{red}}$. Any weight $\check\lambda\in
\Lambda^\vee$ gives rise to a line bundle $\CO(\check\lambda)$ on $\bQ$.

\ssec{indpro}{The embedding $\fQ\hookrightarrow\bQ$}
We fix a coordinate $t$ on $\bC$ such that $t(0)=0,\ t(\infty)=\infty$.
For $\alpha\in\Lambda_+$ we define a $T$-torsor
$\widehat\QM{}^\alpha_\fg\stackrel{p}{\to}\QM^\alpha_\fg$\footnote{Recall that
$\QM^\alpha_\fg$ is {\em reduced} by definition: it is defined as the moduli
scheme cut out by Pl\"ucker relations made reduced.} as follows.
The $\BC$-points of $\widehat\QM{}^\alpha_\fg$ are the collections
$(v_{\check\lambda}\in\CL_{\check\lambda}\subset V_{\check\lambda}\otimes\CO_\bC),\
\check\lambda\in\Lambda^\vee_+$, such that

(a)
$(\CL_{\check\lambda}\subset
V_{\check\lambda}\otimes\CO_\bC)_{\check\lambda\in\Lambda^\vee_+}\in\QM^\alpha_\fg$;
(b) $v_{\check\lambda}\in\Gamma(\bC-0,\CL_{\check\lambda})$ are the
nonvanishing sections satisfying the Pl\"ucker equations.

The projection $p$ forgets the sections $v_{\check\lambda}$.
The action of $T$ on $\widehat\QM{}^\alpha_\fg$ is defined as follows:
$h(v_{\check\lambda}\in\CL_{\check\lambda})=
(\check\lambda(h)v_{\check\lambda}\in\CL_{\check\lambda})$.

Taking a formal expansion of $v_{\check\lambda}$ at $\infty\in\bC$ we obtain
a closed embedding $s_\alpha:\ \widehat\QM{}^\alpha_\fg\hookrightarrow
\widehat\bQ$. Clearly, $s_\alpha$ is $T$-equivariant, and gives rise to the
same named closed embedding $s_\alpha:\ \QM^\alpha_\fg\hookrightarrow\bQ$.
Evidently, for $\beta\geq\alpha$ we have
$s_\alpha=s_\beta\circ\varphi_{\alpha,\beta}$. Hence we obtain the closed
embedding $s:\ \fQ\hookrightarrow\bQ$. The restriction of the line bundle
$\CO(\check\lambda)$ on $\bQ$ to $\fQ$ coincides with the line bundle
$\CO(\check\lambda)$ on $\fQ$.

\ssec{open}{Open subschemes $\fQ_\infty\subset\fQ$ and $\bQ_\infty\subset\bQ$}
We define an open subscheme
$\oQM{}^\alpha_\fg\subset\QM^\alpha_\fg$ formed by all the
quasimaps without defect at $\infty\in\bC$. Clearly,
$\varphi_{\alpha,\beta}(\oQM{}^\alpha_\fg)\subset
\oQM{}^\beta_\fg$. The direct limit of this system is denoted by
$\fQ_\infty$; it is an open sub ind-scheme of $\fQ$.

Note that $s(\fQ_\infty)\subset G(R)/T\cdot U_-(R)\subset\bQ$. We are going to denote the
open subscheme $G(R)/T\cdot U_-(R)\subset\bQ$ by $\bQ_\infty$. For $n\geq1$,
we have a natural projection $p_n:\ \bQ_\infty\to G/U_-(R_n)/T=:\bQ_n$.

\begin{lemma}
\label{codim2}
The restriction
$\Gamma(\fQ,\CO(\check\lambda))\to\Gamma(\fQ_\infty,\CO(\check\lambda))$
is an isomorphism for any $\check\lambda\in\Lambda^\vee$.
\end{lemma}

\begin{proof}
It suffices to prove that the restriction
$\Gamma(\QM^\alpha_\fg,\CO(\check\lambda))\to
\Gamma(\oQM{}^\alpha_\fg,\CO(\check\lambda))$ is an isomorphism
for any $\alpha\in\Lambda_+$. Since the complement of
$\oQM{}^\alpha_\fg$ in $\QM^\alpha_\fg$ has codimension two, it
suffices to know that $\QM^\alpha_\fg$ is normal. However, locally in the
\'etale topology, $\QM^\alpha_\fg$ is isomorphic to the product of the
Zastava space $Z^\alpha_\fg$ and the flag variety $\CB_\fg$. Finally, the
normality of $Z^\alpha_\fg$ is proved in~\cite[Corollary~2.10]{BF11}.
\end{proof}

The following conjecture is not needed in this paper, but it might be useful for future purposes.
\conj{codim two}
The restriction $\Gamma(\bQ,\CO(\check\lambda))\to
\Gamma(\bQ_\infty,\CO(\check\lambda))$ is an isomorphism
for any $\check\lambda\in\Lambda^\vee$.
\econj

Let us make a few remarks about \refco{codim two}.
As in the proof of~Lemma~\ref{codim2}, it suffices to know that the scheme
$\bQ$ is normal. According to~\cite{D},~\cite{GK}, the formal completion of
$\bQ$ at a closed point $x\in\bQ$ is isomorphic to the product of the formal
completion of a certain $\QM^\alpha_\fg$ at a closed point
$\phi\in\QM^\alpha_\fg$, and countably many copies of the formal disc.
So the normality of the formal neighborhood of every closed point follows from the normality of $\QM^\alpha_\fg$.
Unfortunately, since $\bQ$ is not noetherian it does not imply the normality of $\bQ$ itself.

\medskip
\noindent
The group $\BG_m$ acts on $\fQ$ and $\bQ$ by loop rotations, and the line
bundles $\CO(\check\lambda)$ are $\BG_m$-equivariant. Hence $\BG_m$ acts
on the global sections of these line bundles. We will denote by
$\widetilde\Gamma(\fQ,\CO(\check\lambda))\subset\Gamma(\fQ,\CO(\check\lambda))$
the subspace of $\BG_m$-finite sections.

\th{isoq}
The restriction $\Gamma(\bQ_\infty,\CO(\check\lambda))\to
\widetilde\Gamma(\fQ_\infty,\CO(\check\lambda))=
\widetilde\Gamma(\fQ,\CO(\check\lambda))$ is an isomorphism
for any $\check\lambda\in\Lambda^\vee$.
\eth

\begin{proof}
The closed embedding $\varphi_{\alpha,\beta}:\ \QM^\alpha_\fg\hookrightarrow
\QM^\beta_\fg$ lifts in an evident way to the same named closed embedding of
$T$-torsors $\widehat\QM{}^\alpha_\fg\hookrightarrow\widehat\QM{}^\beta_\fg$.
We denote the limit of this system by $\widehat\fQ$, a $T$-torsor over $\fQ$.
The construction of~\refss{indpro} defines a $T$-equivariant closed embedding
$s:\ \widehat\fQ_\infty\hookrightarrow\widehat\bQ_\infty:=G/U_-(R)$. 
We have to prove
that the restriction $\BC[\widehat\bQ_\infty]\to
\widetilde\BC[\widehat\fQ_\infty]=\widetilde\BC[\widehat\fQ]$ is an isomorphism.
Here $\widetilde\BC[\widehat\fQ_\infty]$ (resp. $\widetilde\BC[\widehat\fQ]$)
stands for the ring of $\BG_m$-finite functions on $\widehat\fQ_\infty$
(resp. $\widehat\fQ$).

To this end we mimick the argument of~\cite[Section~2]{BF11}.
We choose a {\em regular} dominant $\mu\in\Lambda^+$, and consider the
corresponding $T$-fixed point $t^\mu\in\Gr_G$. Its stabilizer $\on{St}_\mu$
in $G[t^{-1}]$ has the unipotent radical $\on{RadSt}_\mu$, and the quotient
$\on{St}_\mu/\on{RadSt}_\mu$ is canonically isomorphic to $T$.
The quotient $G[t^{-1}]/\on{St}_\mu$ is the $G[t^{-1}]$-orbit
$\sW_{G,\mu}\subset\Gr_G$ of $t^\mu$ (see~\cite[Section~2.4]{BF11}),
and the quotient $G[t^{-1}]/\on{RadSt}_\mu$ is a $T$-torsor
$\widehat\sW_{G,\mu}$.

{\em NB:} The group denoted $\on{St}_\mu$ in~\cite[Section~2.6]{BF11} is
the intersection of our present $\on{St}_\mu$ with the first congruence
subgroup $G_1\subset G[t^{-1}]$.

In modular terms, $\sW_{G,\mu}$ parametrizes the $G$-bundles on $\bC$
of isomorphism type $W\mu$ equipped with a trivialization on $\bC-0$
(see~\cite[Proof of~Theorem~2.8]{BF11}). Such a bundle $\CF_G$
possesses a canonical Harder-Narasimhan flag $HN(\CF_G)$. Note that this flag
is complete, i.e. it is a reduction to the Borel, since $\mu$ is regular.
In particular, the fiber $\CF_{G,\infty}$ of $\CF_G$ at $\infty\in\bC$ is
equipped with a
canonical reduction to the Borel. Now $\widehat\sW_{G,\mu}$ parametrizes
the data as above along with a further reduction of $\CF_{G,\infty}$ to the
unipotent radical of the Borel.

In complete similarity with~\cite[Lemma~2.7]{BF11} we have

\begin{lemma}
\label{2.7}
(1) Fix $n\geq1$, and let $\mu\in\Lambda^+_{\on{reg}}$ satisfy the following
condition: $\langle\mu,\check\alpha\rangle\geq n$ for every positive root
$\check\alpha$ of $\fg$. Then the image of $\on{RadSt}_\mu$ in
$G[t^{-1}]/G_n=G(R_n)$ is equal to $U_-(R_n)$. In particular, we have a
natural map $\pi_{\mu,n}:\ \widehat\sW_{G,\mu}\to G(R_n)/U_-(R_n)$.

(2) Under the assumption of (1), for every $k<n$, the map
$\pi_{\mu,n}^*:\ \BC[G(R_n)/U_-(R_n)]\to\BC[\widehat\sW_{G,\mu}]$ induces
an isomorphism on functions of homogeneity degree $k$ with respect to $\BG_m$.
\qed
\end{lemma}

We denote the intersection of $\sW_{G,\mu}\subset\Gr_G$ with
$\overline\Gr{}^\lambda_G$ by $\overline\sW{}^\lambda_{G,\mu}$. We denote the
preimage of $\overline\sW{}^\lambda_{G,\mu}\subset\sW_{G,\mu}$ in
$\widehat\sW_{G,\mu}$ by $\widehat\sW{}^\lambda_{G,\mu}$.
In complete similarity with~\cite[Theorem~2.8]{BF11} we have

\begin{lemma}
\label{2.8}
(1) Let $\lambda\geq\mu\in\Lambda^+_{\on{reg}}$, and let $\alpha=\lambda-\mu$.
Then there exists a natural birational $T\times\BG_m$-equivariant morphism
$s^\lambda_\mu:\ \widehat\sW{}^\lambda_{G,\mu}\to\widehat\oQM{}^{-w_0\alpha}_\fg$
such that for any $n$ satisfying the condition in~Lemma~\ref{2.7}(1), the
following diagram is commutative:
\eq{CD}
\begin{CD}
\widehat\sW{}^\lambda_{G,\mu} @>s^{\lam}_{\mu}>> \widehat\oQM{}^{-w_0\alpha}_\fg\\
@V\pi_{\mu,n}VV @VVp_n\circ s_{-w_0\alpha} V\\
 G(R_n)/U_-(R_n)@>\on{id}>> G(R_n)/U_-(R_n)
 \end{CD}
\end{equation}
($s_\alpha$ was constructed in~\refss{indpro}).

(2) The map $(s^\lambda_\mu)^*:\ \BC[\widehat\oQM{}^{-w_0\alpha}_\fg]\to
\BC[\widehat\sW{}^\lambda_{G,\mu}]$ induces an isomorphism on functions of
degree $<n$ for any $n$ satisfying the condition in~Lemma~\ref{2.7}(1). \qed
\end{lemma}

Now~\reft{isoq} immediately follows from~Lemma~\ref{2.7} and~Lemma~\ref{2.8}.
\end{proof}

Note that if one assumes \refco{codim two} then it follows that
the restriction $\Gamma(\bQ,\CO(\check\lambda))\to
\widetilde\Gamma(\fQ,\CO(\check\lambda))$ is an isomorphism
for any $\check\lambda\in\Lambda^\vee$.
(this follows immediately from~\reft{isoq},~Lemma~\ref{codim2}, and~\refco{codim two}).
%--------------------------------------------------------------------------------------------------------

\sec{3}{Cohomology vanishing}

From now on we assume that $G$ is simply laced.

\ssec{noname}{}
The group $\BG_m$ acts on $\fQ$ and $\bQ$ by loop rotations, and the line
bundles $\CO(\check\lambda)$ are $\BG_m$-equivariant. Hence $\BG_m$ acts
on the cohomology $H^n(\fQ,\CO(\check\lambda)):=\displaystyle{\lim_\leftarrow
H^n(\QM^\alpha_\fg,\CO(\check\lambda))}$
of these line bundles. We will denote by
$\widetilde{H}{}^n(\fQ,\CO(\check\lambda))\subset H^n(\fQ,\CO(\check\lambda))$
the subspace of $\BG_m$-finite classes.
%By definition, $H^n(\fQ,\CO(\check\lambda))=\displaystyle{\lim_\leftarrow
%H^n(\QM^\alpha_\fg,\CO(\check\lambda))}$.

Recall that $\alpha\mapsto\alpha^*$ stands for the natural (linear) isomorphism
between the coroot lattice of $\fg$ and its root lattice, taking the simple
coroots to the corresponding simple roots. Now $\Lambda_+$ contains a cofinal
subsystem $\Lambda_+^{\check\lambda}$ formed by $\alpha$ such that
$\alpha^*+\check\lambda$ is dominant.

\th{vanish}
(1) For $n>0$ and $\alpha\in\Lambda_+^{\check\lambda}$ we have
$H^n(\QM^\alpha_\fg,\CO(\check\lambda))=0$.

(2) For $n>0$ and $\check\lambda\in\Lambda^\vee$ we have
$\widetilde{H}{}^n(\fQ,\CO(\check\lambda))=0$.

(3) For $\check\lambda\not\in\Lambda^\vee_+$ we have
$\widetilde{H}{}^0(\fQ,\CO(\check\lambda))=0$.
\eth

\begin{proof}
(3) is clear, and (2) follows from (1). We prove (1).

According to~\cite[Proposition~5.1]{BF11}, $Z^\alpha_\fg$ is a Gorenstein
variety with rational singularities. Since $\QM^\alpha_\fg$ is, locally in
\'etale topology, isomorphic to $Z^\alpha_\fg\times\CB_\fg$, we conclude that
$\QM^\alpha_\fg$ is a Gorenstein variety with rational singularities as well.
(It is here that we use the assumption that $G$ is simply laced.)
Let us denote the dualizing sheaf of $\QM^\alpha_\fg$ by $\omega^\alpha$.

\begin{lemma}
\label{givlee}
$\omega^\alpha\simeq\CO(-\alpha^*-2\check\rho)$.
\end{lemma}

\begin{proof}
In case $G=\SL(N)$, the lemma is proved in~\cite[Theorem~3]{GL}.
For arbitrary simply laced $G$ we first prove that $\omega^\alpha\simeq
\CO(\check\lambda)$ for some $\check\lambda$. It is enough to check this
on the open subscheme $\oQM{}^\alpha_\fg$ since the complement
is of codimension two. We have the morphism of evaluation at $\infty\in\bC:\
\oQM{}^\alpha_\fg\stackrel{ev_\infty}{\longrightarrow}\CB_\fg$.
It is a $G$-equivariant fibration with fibers isomorphic to $Z^\alpha_\fg$.
Since the big cell $U\cdot e_-\subset\CB_\fg$ is a free orbit of $U$, we have
$ev_\infty^{-1}(U\cdot e_-)\simeq Z^\alpha_\fg\times U$. The canonical class
of $Z^\alpha_\fg$ is trivial (see~\cite[Proof of~Proposition~5.1]{BF11}), hence
the canonical class of $ev_\infty^{-1}(U\cdot e_-)$ is trivial as well.
Thus $\omega^\alpha$ has a nowhere vanishing section $\sigma$ on
$ev_\infty^{-1}(U\cdot e_-)$. Hence the class of $\omega^\alpha$ on
$\oQM{}^\alpha_\fg$ is a linear combination of the pullbacks
under $ev_\infty$ of the Schubert divisors on $\CB_\fg$. The pullback of an
irreducible Schubert divisor being $\CO(\check\omega_i)$ we conclude that
there exists $\check\lambda$ such that $\omega^\alpha\simeq\CO(\check\lambda)$.

It remains to check $\check\lambda=-\alpha^*-2\check\rho$. We will do this on
another open subscheme $\overset{\bullet}\QM{}^\alpha_\fg\subset\QM^\alpha_\fg$ with the
complement of codimension two. Namely, $\overset{\bullet}\QM{}^\alpha_\fg$
is the moduli space of quasimaps with defect at most a simple coroot (or no
defect at all). Note that $\overset{\bullet}\QM{}^\alpha_\fg$ is smooth,
and the Kontsevich resolution is an isomorphism over it. Let us fix a
quasimap without defect $\phi\in\QM^{\alpha-\alpha_i}_\fg$, and consider a curve
$C^\phi_i\subset\overset{\bullet}\QM{}^\alpha_\fg$ formed by all the
quasimaps $\phi(\alpha_i\cdot c),\ c\in\bC$ (twisting $\phi$ by an arbitrary point of $\bC$).
It is easy to see that $\deg(\CO(\check\omega_j)|_{C^\phi_i})=\delta_{ij}=
\langle\alpha_i,\check\omega_j\rangle$. Hence it remains to check that
$\deg(\omega^\alpha|_{C^\phi_i})=-\langle\alpha_i,\alpha^*+2\check\rho\rangle$.
This is done in~\cite[Proposition~4.4]{FK}. Although {\em loc. cit.} is
formulated for $G=\SL(N)$, its proof goes through word for word for arbitrary
simple $G$.

The lemma is proved. \end{proof}

We are ready to finish the proof of the theorem. For
$\alpha\in\Lambda^{\check\lambda}_+$ the line bundle
$\CL=\CO(\check\lambda)\otimes(\omega^\alpha)^*$ on $\QM^\alpha_\fg$ is
very ample. We have to prove that $H^n(\QM^\alpha_\fg,\CO(\check\lambda))=
H^n(\QM^\alpha_\fg,\CL\otimes\omega^\alpha)=0$ for $n>0$. According
to~\cite[Proposition~5.1]{BF11}, $\QM^\alpha_\fg$ has rational singularities.
Let $\pi:\ X\to\QM^\alpha_\fg$ be a resolution of singularities. Then for the
canonical line bundle $\omega_X$ of $X$ we have $R\pi_*\omega_X=\omega^\alpha$.
Hence $H^n(\QM^\alpha_\fg,\CL\otimes\omega^\alpha)=
H^n(X,\pi^*\CL\otimes\omega_X)=0$ (for $n>0$) by Kawamata-Viehweg vanishing
since $\pi^*\CL$ is nef and big.

This completes the proof of the theorem.
\end{proof}

\sec{4}{$q$-Whittaker functions}

\ssec{charac}{The character of $R\Gamma(\QM^\alpha_\fg,\CO(\check\lambda))$}
Recall~\cite[Introduction]{BF11} that $\fJ_\alpha(q,z)$
is the character of $T\times\BG_m$-module $\BC[Z^\alpha_\fg]$, a rational
function on $T\times\BG_m$. Let $x_i$ stand for the character of the dual
torus $\check T$ corresponding to the simple coroot $\alpha_i$. For
$\alpha\in\Lambda_+$ the corresponding character of $\check T$ is denoted by
$x^\alpha$. We consider the formal generating functions
$J_\fg(q,z,x)=\sum_{\alpha\in\Lambda_+}x^\alpha\fJ_\alpha$, and
$\fJ_\fg(q,z,x)=\prod_{i\in I}x_i^{\log(\check\omega_i)/\log q}J_\fg(q,z,x)$,
cf.~\cite[Equation~(18)]{BF0}.

According to~\cite[Corollary~1.6]{BF11}, the function $\fJ_\fg(q,z,x)$
is an eigen-function of the quantum difference Toda integrable system
associated with $\fg$. For example, if $G=\SL(N)$, the function
$\fJ_\fg(q,z,x)$ is an eigenfunction of the operator
$\fG=T_1+T_2(1-x_1)+\ldots+T_N(1-x_{N-1})$, cf.~\cite[Equation~(16)]{BF0},
where $T_k(F(q,z,x_1,\ldots,x_{N-1}))=F(q,z,x_1,\ldots,x_{k-2},q^{-1}x_{k-1},
qx_k,x_{k+1},\ldots,x_{N-1})$.

Note that if we plug $x=q^{\check\lambda}$ into $J_\fg(q^{-1},z,x)$ or into
$\fJ_\fg(q^{-1},z,x)$, then for $\check\lambda\in\Lambda^\vee_+$ these formal
series converge, and we have $\fJ_\fg(q^{-1},z,q^{\check\lambda}):=\prod_{i\in I}
(q^{\langle\alpha_i,\check\lambda\rangle})^{\log(\check\omega_i)/\log q}
J_\fg(q^{-1},z,q^{\check\lambda})=
z^{\check\lambda}J_\fg(q^{-1},z,q^{\check\lambda})$ (a formal Taylor series in
$q$ with coefficients in Laurent polynomials in $z$).

The following lemma is a reformulation of~\cite[Proposition~2]{GL}:

\begin{lemma}
\label{gile}
The class of $R\Gamma(\QM^\alpha_\fg,\CO(\check\lambda))$ in
$K_{T\times\BG_m}(pt)$ equals
$$\sum_{\substack{\gamma+\beta=\alpha\\ w\in W}}
z^{w\check\lambda}q^{\langle\gamma,\check\lambda\rangle}
\fJ_\gamma(q^{-1},wz^{-1})\fJ_\beta(q,wz^{-1})
\prod_{\check\alpha\in\check{R}{}^+}(1-wz^{-\check\alpha})^{-1}.$$
\end{lemma}

\begin{proof}
Let $\pi:\ \overline{M}_{0,0}(\BP^1\times\CB_\fg,(1,\alpha))\to\QM^\alpha_\fg$
(resp. $\varpi:\ M^\alpha_\fg\to Z^\alpha_\fg$) be the Kontsevich resolution,
see e.g.~\cite[Appendix]{FFKM} (resp.~\cite[Proof of~Proposition~5.1]{BF11}).
Since the singularities of $\QM^\alpha_\fg$ (resp. $Z^\alpha_\fg$) are rational,
we have $R\Gamma(\QM^\alpha_\fg,\CO(\check\lambda))=
R\Gamma(\overline{M}_{0,0}(\BP^1\times\CB_\fg,(1,\alpha)),
\pi^*\CO(\check\lambda))$ (resp. $\BC[Z^\alpha_\fg]=\BC[M^\alpha_\fg]$).
Hence we have to express the character of
$R\Gamma(\QM^\alpha_\fg,\CO(\check\lambda))$ via the characters of
$\BC[M^\beta_\fg]$. This is done in~\cite[Proof of~Proposition~2]{GL} via
the Atiyah-Bott-Lefschetz localization to $T\times\BG_m$-fixed points of
$\overline{M}_{0,0}(\BP^1\times\CB_\fg,(1,\alpha))$. As usually, we have to
add that {\em loc. cit.} deals with $G=\SL(N)$, however, the proof
goes through word for word for arbitrary semisimple $G$.
\end{proof}

\ssec{lim charac}{The character of $\widetilde\Gamma(\fQ,\CO(\check\lambda))$}
By the proof of~\reft{isoq} and~Lemma~\ref{2.8}(2), the character
$\chi(\widetilde\Gamma(\fQ,\CO(\check\lambda)))$ is the limit of the characters
$\chi (R^0\Gamma(\QM^\alpha_\fg,\CO(\check\lambda)))$
as $\alpha\to\infty$. By~\reft{vanish}(1), as $\alpha\to\infty$, the limit
of the characters $\chi (R^{>0}\Gamma(\QM^\alpha_\fg,\CO(\check\lambda)))$
vanishes. Thus, the character $\chi(\widetilde\Gamma(\fQ,\CO(\check\lambda)))$
is the limit of the characters $\chi (R\Gamma(\QM^\alpha_\fg,\CO(\check\lambda)))$
as $\alpha\to\infty$. We define
$\fJ_\infty(q,z):=\displaystyle{\lim_{\alpha\to\infty}\fJ_\alpha(q,z)}$
(it is easy to see that the latter limit exists).

\prop{brav}
$$
\chi(\widetilde\Gamma(\fQ,\CO(\check\lambda)))=
\sum_{w\in W}\fJ_\fg(q^{-1},wz^{-1},q^{\check\lambda})\fJ_\infty(q,wz^{-1})
\prod_{\check\alpha\in\check{R}{}^+}(1-wz^{-\check\alpha})^{-1}.
$$
\eprop

\begin{proof}
As $\alpha$ goes to $\infty$, the formula of~Lemma~\ref{gile} goes to
$$\sum_{\substack{\gamma\in\Lambda_+\\ w\in W}}
z^{w\check\lambda}q^{\langle\gamma,\check\lambda\rangle}
\fJ_\gamma(q^{-1},wz^{-1})\fJ_\infty(q,wz^{-1})
\prod_{\check\alpha\in\check{R}{}^+}(1-wz^{-\check\alpha})^{-1}=$$
$$\sum_{w\in W}
z^{w\check\lambda}J_\fg(q^{-1},wz^{-1},q^{\check\lambda})\fJ_\infty(q,wz^{-1})
\prod_{\check\alpha\in\check{R}{}^+}(1-wz^{-\check\alpha})^{-1}=$$
$$\sum_{w\in W}\fJ_\fg(q^{-1},wz^{-1},q^{\check\lambda})\fJ_\infty(q,wz^{-1})
\prod_{\check\alpha\in\check{R}{}^+}(1-wz^{-\check\alpha})^{-1}.$$
\end{proof}

\cor{gera}
Let $\chi(\widetilde\Gamma(\fQ,\CO(\check\lambda)))=\Psi_{\chlam}(q,z)$. Then the functions
$\Psi_{\chlam}(q,z)$ satisfy all the conditions of \refco{whit-exist-co}.
\ecor

\begin{proof}
Part 2 of \refco{whit-exist-co} is obvious by construction. Also \refco{whit-exist-co}(1b) is obvious.
According to~\reft{vanish}(2), $\chi(\widetilde\Gamma(\fQ,\CO(\check\lambda)))=0$
if $\check\lambda\not\in\Lambda^\vee_+$, which proves \refco{whit-exist-co}(1a).

Let us prove \refco{whit-exist-co}(1c). The function
$\fJ_\fg(q^{-1},wz^{-1},q^{\check\lambda})$ on the lattice $\Lambda^\vee$ is an
eigenfunction of the quantum difference Toda restricted to the lattice.
According to~\refp{brav}, $\chi(\widetilde\Gamma(\fQ,\CO(\check\lambda)))$ is
a linear combination of the functions $\fJ_\fg(q^{-1},wz^{-1},q^{\check\lambda})$
with coefficients independent of $\check\lambda$. Hence $\Psi_{\chlam}(q,z)$
is an eigenfunction of the quantum difference Toda as well.
\end{proof}

\sec{5}{Weyl modules}

\ssec{nameless}{} Recall that $R=\BC[[t^{-1}]]$. We introduce a new variable
$\st=t^{-1}$, so that $R=\BC[[\st]]$. We set $\widetilde R:=\BC[\st]\subset R$.
The proalgebraic group $G(R)$ acts naturally on the profinite dimensional
vector space $\Gamma(\fQ,\CO(\check\lambda))$. The continuous dual
$\Gamma(\fQ,\CO(\check\lambda))^\vee$ coincides with the graded dual
$\widetilde\Gamma(\fQ,\CO(\check\lambda))^\vee$, and is equipped with a
natural action of $G(\widetilde R):\ g\cdot v^*(v):=v^*(\tau g\cdot v)$.
Here $g\mapsto\tau g$ is the Chevalley antiinvolution of $G$ identical on $T$.
The derivative of these actions gives rise to the actions of $\fg(R)$
and $\fg(\widetilde R)$. According to~\reft{isoq}, the
$\fg(\widetilde R)$-module $\Gamma(\fQ,\CO(\check\lambda))^\vee$ coincides
with the graded dual $\Gamma(\bQ_\infty,\CO(\check\lambda))^\vee$.

We denote the preimage of the big cell $U\cdot e_-\subset\CB_\fg$ in
$G/U_-\to\CB_\fg$ by $C\subset G/U_-$. We denote the open subscheme
$C(R)/T\subset G(R)/T\cdot U_-(R)=\bQ_\infty$ by $\overset{\circ}{\bQ}$.
We have the restriction morphism of $\fg(R)$-modules
$\Gamma(\bQ_\infty,\CO(\check\lambda))\hookrightarrow
\Gamma(\overset{\circ}{\bQ},\CO(\check\lambda))$. Now $C(R)$ is a free
orbit of $B(R)\subset G(R)$, and
$\Gamma(\overset{\circ}{\bQ},\CO(\check\lambda))=
\on{CoInd}_{\fu(R)\oplus\ft}^{\fg(R)}\BC_{\check\lambda}$. The graded dual
$\Gamma(\overset{\circ}{\bQ},\CO(\check\lambda))^\vee=
\on{Ind}_{\fu(\widetilde R)\oplus\ft}^{\fg(\widetilde R)}\BC_{\check\lambda}$.

\begin{lemma}
\label{?}
$\Gamma(\bQ_\infty,\CO(\check\lambda))\subset
\Gamma(\overset{\circ}{\bQ},\CO(\check\lambda))$ is the maximal
$G$-integrable $\fg(R)$-submodule. Equivalently,
$\Gamma(\bQ_\infty,\CO(\check\lambda))^\vee$ is the maximal $G$-integrable
$\fg(\widetilde R)$-quotient module of
$\Gamma(\overset{\circ}{\bQ},\CO(\check\lambda))^\vee$.
\end{lemma}

\begin{proof}
Note that $\bQ_\infty$ is the $G$-saturation of $\overset{\circ}{\bQ}$.
Let $v\in\Gamma(\overset{\circ}{\bQ},\CO(\check\lambda))$ lie in a
finite-dimensional $\fg$-submodule
$V\subset\Gamma(\overset{\circ}{\bQ},\CO(\check\lambda))$. The action of
$\fg$ on $V$ integrates to the action of $G$. Let us view $v$ as a
$\check\lambda$-covariant function on $C(R)$. We have to check that
$v$ is the restriction of a $\check\lambda$-covariant function $\hat v$
on $G/U_-(R)$ to $C(R)$. Given a point $y\in G/U_-(R)$ we can find $g\in G$
such that $g(y)\in C(R)$. Then we define $\hat{v}(y):=u(gy)$ where we view
$u:=gv\in V$ as a $\check\lambda$-covariant function on $C(R)$. Clearly, this
is well defined, i.e. independent of a choice of $g$.
\end{proof}

Recall the notion of Weyl $\fg(\widetilde R)$-module $\CW(\check\lambda)$
for dominant $\check\lambda\in\Lambda^\vee_+$, see e.g.~\cite{CFK}.
It is the maximal $G$-integrable $\fg(\widetilde R)$-quotient module of
$\on{Ind}_{\fu(\widetilde R)\oplus\ft}^{\fg(\widetilde R)}\BC_{\check\lambda}$
({\em loc. cit.}). Thus~Lemma~\ref{?} implies the first part of \reft{demazure}.

On the other hand, taking into account \reft{isoq} we also get
\prop{!}
For $\check\lambda\in\Lambda^\vee_+$, we have a natural isomorphism of
$\fg(\widetilde R)$-modules $\Gamma(\fQ,\CO(\check\lambda))^\vee\simeq
\CW(\check\lambda)$. \qed
\eprop

Combining this with \refc{gera} we get the following
\cor{weyl-whit}
$\chi(\calW(\chlam))=\Psi_{\chlam}(q,z)$.
\ecor
This is actually the statement of \reft{demazure}(1).
To prove \reft{demazure}(2) let us
recall that the Demazure module $D(\check\lambda)$ is a certain
$\fg(\widetilde R)$-submodule of an irreducible integrable level one
representation of
$\fg_{\on{aff}}$, see e.g.~\cite[2.2]{FL}. In addition, according
to~\cite{CFK},~\cite{FL} there exists
an action of $\CC[\AA^{\chlam}]$ on $\calW(\chlam)$ such that

1) This action commutes with $G(R)\rtimes \CC^*$.

2) $\calW(\chlam)$ is finitely generated and free
over $\CC[\AA^{\chlam}]$.

3) The fiber of $\calW(\chlam)$ at $\chlam\cdot 0$ is isomorphic to $D(\chlam)$.

\medskip
\noindent
Thus we get the following corollary, which is actually the statement of \reft{demazure}(2)
(as was mentioned in the introduction it was proved in~\cite{GLOIII} for $G=\SL(N)$):
\cor{dem}
The product $\chi(\widetilde\Gamma(\fQ,\CO(\check\lambda)))\cdot\displaystyle
{\prod_{i\in I}\prod_{r=1}^{\langle\alpha_i,\check\lambda\rangle}(1-q^r)}=\hatPsi_{\chlam}(q,z)$
is equal to
the character of the (finite dimensional) Demazure module $D(\check\lambda)$. In particular, it is a finite
linear combination of $\chi(L(\chmu))$'s with coefficients in $\ZZ_{\geq 0}[q]$.
\ecor

%------------------------------------------------------------------------------------------------------------------
\ssec{}{Geometric interpretation of the $\CC[\AA^{\chlam}]$-action}
We conclude the paper by giving an interpretation of the $\CC[\AA^{\chlam}]$-action on $\calW(\chlam)$ in terms of
\reft{demazure}(1). This will enable us to prove the second assertion of \reft{weyl}.
It would be nice to prove that this action is free directly by geometric means (without referring to
\cite{FL}).

Let $T(R)_1$ denote the first congruence subgroup in $T(R)$ (i.e. the kernel of the natural map
$T(R)\to T$). Let $\grt(R)_1$ denote its (abelian) Lie algebra (i.e. the kernel of the natural map
$\grt(R)\to \grt$). We denote by $\grt(\tilR)_1\subset \grt(R)_1$ the corresponding subspace (consisting of all
mappings $\AA^1\to \grt$ which are equal to $0$ at $0$). Then for every $\chlam\in\Lam^{\vee}_+$ there exists a natural epimorphism
$\pi_{\chlam}:U(\grt(\tilR)_1)=\Sym(\grt(\tilR)_1)\to \CC[\AA^{\chlam}]$ defined by the following formula:
$$
\pi_{\chlam} (h\st^n)(\sum_i \gam_i x_i)=\sum_i \la h,\gam_i\rangle x_i^n.
$$
Here $h\in \grt$ and $\sum_i \gam_i x_i\in \AA^{\chlam}$.

Clearly, the group $T(R)_1$ acts (on the right) on the scheme $\bQ_{\infty}=G(R)/T\cdot U_-(R)$. Hence we get a natural action
of $\Sym(\grt(\tilR)_1)\subset\Sym(\grt(R)_1)$ on 
$\Gam(\bQ_{\infty}, \calO(\chlam))$ for every $\chlam\in \Lam^{\vee}$.
The following result is easy to prove; we leave
the details to the reader:
\prop{config}
\begin{enumerate}
\item The above action
of $\Sym(\grt(\tilR)_1)$ on $\Gam(\bQ_{\infty}, \calO(\chlam))$ factors through $\pi_{\chlam}$.
\item
The resulting action
of $\CC[\AA^{\chlam}]$ on $\Gam(\bQ_{\infty}, \calO(\chlam))^{\vee}=\calW(\chlam)$
coincides with the action considered in~\cite{CFK} and \cite{FL}.
\end{enumerate}
\eprop
From \refp{config} we immediately get the following
\cor{}
We have
$\Gam(G(R)/B_-(R),\calO(\chlam))\simeq D(\chlam)^{\vee}$ (this is the second assertion of \reft{weyl}).
\ecor
\begin{proof}
It follows from  \refp{config} and from the fact that $D(\chlam)$ is the fiber of $\calW(\chlam)$ over
$\chlam\cdot 0\in \CC[\AA^{\chlam}]$ that $D(\chlam)^{\vee}$ is isomorphic to the invariants of $\grt(\tilR)$ on $\calW(\chlam)^{\vee}$.
Since $\grt(\tilR)_1$ is dense in $\grt(R)_1$, it follows that
$$
(\calW(\chlam)^{\vee})^{\grt(\tilR)_1}=(\calW(\chlam)^{\vee})^{\grt(R)_1}.
$$
From \refp{!} we get
$$
(\calW(\chlam)^{\vee})^{\grt(R)_1}=\Gam(G(R)/T\cdot U_-(R),\calO(\chlam))^{\grt(R)_1}=\Gam(G(R)/B_-(R),\calO(\chlam)).
$$

\end{proof}

\begin{acknowledgement}
This paper emerged as a result of numerous conversations held between the first author and
A.~Borodin in IHES in June 2011. In particular, A.~Borodin has introduced the first author to the notion of $q$-Whittaker
functions and brought to his attention the papers \cite{GLOI}-\cite{GLOIII}. The first author wishes
also to thank IHES staff for exceptionally pleasant working conditions and hospitality.
We are grateful to B.~Feigin and S.~Loktev for
their numerous patient explanations about Weyl modules. The observation that
$\Gamma(G[[\st]]/T\cdot U_-[[\st]],\CO(\check\lambda))$ is a dual Weyl module is essentially due to them.
We would like to thank K.~Schwede and other mathoverflowers for their
explanations about Kodaira vanishing. The first author is also deeply indebted to Patrick
Clark for his patient explanations about Macdonald polynomials and to P.~Etingof and I.~Cherednik for explanations about the
$q$-Toda system.
M.F. was partially supported by the RFBR grants 12-01-00944, 12-01-33101,
13-01-12401/13,
the National Research University Higher School of Economics' Academic Fund
award No.12-09-0062 and
the AG Laboratory HSE, RF government grant, ag. 11.G34.31.0023.
This study comprises research findings from the ``Representation Theory
in Geometry and in Mathematical Physics" carried out within The
National Research University Higher School of Economics' Academic Fund Program
in 2012, grant No 12-05-0014.
\end{acknowledgement}

%\bigskip
%\footnotesize{
%{\bf A.B.}: Department of Mathematics, Brown University,
%151 Thayer St., Providence RI
%02912, USA;\\
%{\tt braval@math.brown.edu}}

%\footnotesize{
%{\bf M.F.}: IMU, IITP and National Research University
%Higher School of Economics\\
%Department of Mathematics, 20 Myasnitskaya st, Moscow 101000 Russia;\\
%{\tt fnklberg@gmail.com}}

\end{document}